\documentclass[12pt]{amsart}
 \usepackage{amscd,amsmath,amsthm,amssymb,comment}
 \def\NZQ{\Bbb}               
 \def\NN{{\NZQ N}}
 
 \def\ZZ{{\NZQ Z}}

 \def\FF{{\NZQ F}}

 \def\frk{\frak}               

 \def\rr{{\frk r}}
 \def\ss{{\frk s}}

 \def\mm{{\frk m}}
 
 \def\ab{{\bold a}}
 \def\bb{{\bold b}}
 \def\xb{{\bold x}}

 \def\cb{{\bold c}}
 
 \def\eb{{\bold e}}
 \def\tb{{\bold t}}
 \def\qb{{\bold q}}
 \def\opn#1#2{\def#1{\operatorname{#2}}} 
 \opn\chara{char} \opn\length{\ell} \opn\pd{pd} \opn\rk{rk}
 \opn\projdim{proj\,dim} \opn\injdim{inj\,dim} \opn\rank{rank}
 \opn\depth{depth} \opn\grade{grade} \opn\height{height}
 \opn\embdim{emb\,dim} \opn\codim{codim}
 
 \opn\Tr{Tr} \opn\bigrank{big\,rank}
 \opn\superheight{superheight}\opn\lcm{lcm}
 \opn\trdeg{tr\,deg}
 \opn\reg{reg} \opn\lreg{lreg} \opn\ini{in} \opn\lpd{lpd}
 \opn\size{size} \opn\sdepth{sdepth}
 \opn\link{link}\opn\fdepth{fdepth}\opn\lex{lex} \opn\op{op}
 \opn\div{div} \opn\Div{Div} \opn\cl{cl} \opn\Cl{Cl}
 \opn\Spec{Spec} \opn\Supp{Supp} \opn\supp{supp} \opn\Sing{Sing}
 \opn\Ass{Ass} \opn\Min{Min}\opn\Mon{Mon}
 \opn\Ann{Ann} \opn\Rad{Rad} \opn\Soc{Soc}
 \opn\Im{Im} \opn\Ker{Ker} \opn\Coker{Coker} \opn\Am{Am} \opn\ann{ann}
 \opn\Hom{Hom} \opn\Tor{Tor} \opn\Ext{Ext} \opn\End{End}
 \opn\Aut{Aut} \opn\id{id}
 
 \opn\nat{nat}
 \opn\pff{pf}
 \opn\Pf{Pf} \opn\GL{GL} \opn\SL{SL} \opn\mod{mod} \opn\ord{ord}
 \opn\Gin{Gin} \opn\Hilb{Hilb}\opn\sort{sort}
 \opn\aff{aff} \opn
 \con{conv} \opn\relint{relint} \opn\st{st}
 \opn\lk{lk} \opn\cn{cn} \opn\core{core} \opn\vol{vol}
 \opn\link{link} \opn\star{star}\opn\lex{lex}\opn\set{set}
 \opn\gr{gr}
 
 \def\pot#1#2{#1[\kern-0.28ex[#2]\kern-0.28ex]}

 \opn\dirlim{\underrightarrow{\lim}}
 \opn\inivlim{\underleftarrow{\lim}}
 \let\to=\rightarrow
 
 \def\Implies{\ifmmode\Longrightarrow \else
         \unskip${}\Longrightarrow{}$\ignorespaces\fi}
 \def\implies{\ifmmode\Rightarrow \else
         \unskip${}\Rightarrow{}$\ignorespaces\fi}
 \def\iff{\ifmmode\Longleftrightarrow \else
         \unskip${}\Longleftrightarrow{}$\ignorespaces\fi}

 \let\:=\colon
 \newtheorem{Theorem}{Theorem}[section]
 \newtheorem{Lemma}[Theorem]{Lemma}
 \newtheorem{Corollary}[Theorem]{Corollary}
 \newtheorem{Proposition}[Theorem]{Proposition}
 \newtheorem{Remark}[Theorem]{Remark}
 
 \newtheorem{Example}[Theorem]{Example}

 \let\epsilon\varepsilon
 \let\kappa=\varkappa
 \def\qed{\ifhmode\textqed\fi
       \ifmmode\ifinner\quad\qedsymbol\else\dispqed\fi\fi}
 \def\textqed{\unskip\nobreak\penalty50
        \hskip2em\hbox{}\nobreak\hfil\qedsymbol
        \parfillskip=0pt \finalhyphendemerits=0}
 \def\dispqed{\rlap{\qquad\qedsymbol}}
 
 \opn\dis{dis}
 \def\pnt{{\raise0.5mm\hbox{\large\bf.}}}
 
 \opn\Lex{Lex}

 \def\ModN{{\bf  Mod}_{S}^{\mathbb{N}^n}}
 \def\ModZ{{\bf  Mod}_{S}^{\mathbb{Z}^n}}
 \def\One{{\bf 1}}
\newcommand\Zero{{\bf 0}}

\newcommand\Modtb{{\bf  Mod}_{S}^{\tb}}
\newcommand\Modone{{\bf Mod}_S^\One}
\newcommand\grHom{\underline{\Hom}}
\newcommand\grExt{\underline{\Ext}}
\newcommand\Dc{\mathcal{D}}
\newcommand\Ac{\mathcal{A}}
\newcommand\DD{\mathbf{D}}
\newcommand\longto{\longrightarrow}
\opn\Proj{Proj} \opn\Inj{Inj}

\begin{document}

 \title {On the radical of multigraded modules}
 \author {Viviana Ene}
 \address{Faculty of Mathematics and Computer Science, Ovidius University, Bd. Mamaia 124, 900527 Constanta, Romania
   \newline \indent Institute of Mathematics of the Romanian Academy, P.O. Box 1-764, RO-014700, Buchaest, Romania}
 \email{vivian@univ-ovidius.ro}
 \author{Ryota Okazaki}
 \thanks{The first author was partially supported  by the JSPS Invitation Fellowship Programs for Research in Japan and by  the grant UEFISCDI,  PN-II-ID-PCE- 2011-3-1023. The second author is partially supported by JST, CREST and also by KAKENHI (no. 20624109)}
 \address{Department of Pure and Applied Mathematics, 
Graduate School of Information Science and Technology,
Osaka University, Toyonaka, Osaka 
560-0043, Japan}
 \email{okazaki@ist.osaka-u.ac.jp}

 \begin{abstract}
We define a functor $\rr^\ast$ from the category of positively determined modules to the category of squarefree modules which plays the role of passing from a monomial ideal to its radical. By using this functor, we generalize several  results on properties that are shared by a monomial ideal and its radical. Moreover, we study the connection of $\rr^\ast$ to the Alexander duality and Auslander-Reiten translate functor. 
 \end{abstract}
 


 \maketitle

 \section*{Introduction}
 
 Our original motivation for this work was to generalize some of the results obtained by Herzog, Takayama, and Terai in \cite{HTT}. They proved that 
 many properties of a monomial ideal pass to its radical. It is well known that every monomial ideal $I$ in the polynomial ring $S=K[x_1,\ldots,x_n]$ 
 over a field $K$ is a positively $\tb$-determined $S$-module for an appropriate $\tb\in \NN^n$ as it was defined by Miller in \cite{Mill}. Thus, a 
 natural way to generalize the results of \cite{HTT} is to consider positively determined modules  instead of monomial ideals. We show that one may 
 define a functor $\rr^\ast$ from the category $\Modtb$ of positively $\tb$-determined modules to the category $\Modone$ of squarefree $S$-modules  
 which plays the role of passing from a monomial ideal to its radical. As it was shown in \cite{BF}, to any ordering preserving map $\qb:\NN^n\to 
 \NN^n$, one may associate a functor $\qb^\ast$ from the category of $\NN^n$-graded $S$-modules to itself which is defined as follows. For any $M\in 
 \ModN$ and $\ab\in \NN^n,$ $(\qb^\ast M)_{\ab}=M_{q(\ab)}$ and the $S$-module structure of $\qb^\ast M$ is given by the multiplication $(q^\ast 
 M)_{\bb}\stackrel{\cdot\xb^\ab}{\rightarrow}(q^\ast M)_{\bb+\ab}$ that maps every homogeneous element $y\in (q^\ast M)_{\bb}$ to 
$\xb^{q(\ab+\bb)-q(\bb)}y.$ If $f:M\to N$ is a graded morphism of $\NN^n$-graded $S$-modules, then 
for every $\ab\in \NN^n,$ the $\ab$-degree component of $q^\ast f: q^\ast M \to q^\ast N$ is $f_{q(\ab)}.$
In \cite{BF} it is shown that $q^\ast$ is an exact functor. In Section~\ref{radicalsection}, we have considered the following map. For $\tb\in \NN^n$ with $\tb\geq \One$ where $\One=(1,\ldots,1), $ we define $\rr: \NN^n \to \NN^n$  by $\rr(\ab)=(r_i(a_i))_{1\leq i\leq n}$ where 
\[
r_i(a_i)=\left\{ 
\begin{array}{ll}
	t_i, & \text{ if } a_i>0,\\
	0, &  \text{ otherwise. }
\end{array}\right.
\] This is an ordering preserving map which induces a functor $\rr^\ast$ which depends on $\tb$ from the category $\ModN$ to itself. We showed in Section~\ref{radicalsection} that this functor transports the category $\Modtb$ into the category of squrefree modules and, moreover, for any monomial ideal $I\subset S,$ we have $\rr^\ast I\cong \sqrt{I}$ as $S$-modules. 

As it was explained in \cite{HTT}, the Betti numbers do not increase when one passes from a monomial ideal to its radical. We show in 
Theorem~\ref{bettiradical} that passing from a positively $\tb$-determined module to its ''radical'' module has a similar behavior. In particular,  one 
obtains $\depth M\leq \depth\rr^\ast M$ for any $M\in \Modtb$. 

Unlike the monomial case, for a positively $\tb$-determined module $M,$ we show in Corollary~\ref{dim}, that one has only the inequality $\dim 
\rr^\ast M\leq \dim M.$ Easy examples show that the inequality may be strict. By using the inequalities between depth and Krull dimension, we show in 
Theorem~\ref{CM} that the (sequentailly) Cohen-Macaulay property of $M$ passes to the ''radical'' of $M$ for any postively $\tb$-determined module $M$with $\rr^\ast M\neq 0.$

In Section~\ref{extsection} we study the connection between the functor $\rr^\ast$ and $\Ext$. The main result of the section is Theorem~\ref{ext} 
which states that for every module $M\in \Modtb$ there exists a natural isomorphism $
\grExt^p_S(M, \omega_S)_{\geq \Zero} \cong \grExt^p_S(\rr^\ast(M),\omega_S)$ for all $p,$ where $\omega_S$ is the canonical module of $S.$  In 
particular, for any Cohen-Macaulay ideal $I\subset S$ such that $S/I$ is a positively $\tb$-determined module, it follows that the canonical module
of $S/\sqrt{I}$ is isomorphic to the positive part of $\omega_{S/I}.$  From Theorem~\ref{ext}, under an additional  condition, it follows that if 
$M\in \Modtb$ is a  generalized Cohen-Macaulay (Buchsbaum) module, then $\rr^\ast M$ shares the same property.

Finally, in the last two sections we show how our ''radical'' functor is connected to the Alexander duality (Proposition~\ref{prop:r_and_A}) and Auslander-Reiten translate functor (Proposition~\ref{ART}).

\section*{Acknowledgment}

A part of this research was done when the second author visited the Faculty of Mathematics and Informatics of Ovidius University.
He is deeply grateful for warm hospitality.

 \section{The $\rr^\ast$ functor and first applications}
 \label{radicalsection}
 
In this section we define the $\rr^\ast$ functor on the category $\ModN$ of the $\NN^n$-graded $S$-modules where $S=K[x_1,\ldots,x_n]$ is the 
polynomial ring in $n$ variables over a field $K.$
We first recall the basic notions and set the notation.
Let  $\NN$ be the set of non-negative integers.
For $\ab=(a_1,\ldots,a_n)\in \NN^n$ we set $\xb^\ab=x_1^{a_1}\cdots x_n^{a_n}$ and call $\ab$ the degree of the monomial  $\xb^\ab$. $\nu_i(u)$ denotes the exponent of  variable $x_i$ in the monomial $u\in S$.  Let 
 $\leq$ be the partial order on $\ZZ^n$ which is defined as follows. If $\ab=(a_1,\ldots,a_n),\bb=(b_1,\ldots,b_n)\in\ZZ^n$, then $\ab\leq \bb$ 
if $a_i\leq b_i$  for $1\leq i\leq n.$ Of course, this order induces a partial order on $\NN^n.$

Let $\tb\in \NN^n$ with $\tb\geq \One,$ where $\One=(1,\ldots,1).$ According to \cite{Mill}, a $\ZZ^n$-graded $S$-module $M$ is called
{\em positively $\tb$-determined} if it is finitely generated, $\NN^n$-graded,
and if the multiplication map $M_{\ab}\stackrel{x_i}{\rightarrow}{M_{\ab+\eb_i}}$ is
an isomorphism of $K$-vector spaces whenever $a_i \ge t_i$. Here, $\eb_i$ 
is the vector of $\ZZ^n$ with its $i$-th component equal to $1$ and all the others equal to $0.$ 
A monomial ideal $I$ is positively $\tb$-determined if and only if
it is generated by some elements $x^\ab$ with $\Zero \le \ab \le \tb$.
Every finitely generated $\NN^n$-graded $S$-module
is positively $\tb$-determined for some $\tb \gg \One$.
In particular, for any 2 monomial ideals $I,J$ with $J \supseteq I$,
$J/I$ is positively $\tb$-determined for some $\tb \gg \One$.

Let $\Modtb$ be the full subcategory of $\ModZ$
consisting of positively $\tb$-determined $S$-modules.

According to \cite{BF}, with any order preserving map $q:\ZZ^n\to \ZZ^n,$ one may associate a functor $q^{\ast}: \ModZ \to \ModZ$. Since we are 
concerned only with $\NN^n$-graded modules, that is, $\ZZ^n$-graded modules whose components of degree $\ab\in \ZZ^n\setminus \NN^n$ are all zero, we 
may consider the map $q:\NN^n\to \NN^n.$

$q^\ast$ acts on modules and morphisms as follows. For a  $\ZZ^n$-graded $S$-module $M,$ the $\ab$-degree component of $q^\ast M$ is $(q^\ast 
M)_{\ab}=M_{q(\ab)}.$ The multiplication which gives the $S$-module structure of $q^\ast M$ is the following. For a monomial $\xb^\ab\in S,$ the map 
$(q^\ast M)_{\bb}\stackrel{\cdot\xb^\ab}{\rightarrow}(q^\ast M)_{\bb+\ab}$ maps every homogeneous element $y\in (q^\ast M)_{\bb}$ to 
$\xb^{q(\ab+\bb)-q(\bb)}y.$

We describe now the action of $q^\ast$ on the morphisms of the category $\ModZ$ following \cite{BF}. If $f:M\to N$ is a graded morphism of $\ZZ^n$-graded $S$-modules, then 
for every $\ab\in \ZZ^n,$ the $\ab$-degree component of $q^\ast f: q^\ast M \to q^\ast N$ is $f_{q(\ab)}.$
In \cite{BF} it is shown that $q^\ast$ is an exact functor.

We consider now the following order preserving map. Let $\tb \in \NN^n, \tb\geq \One,$ and $\rr: \NN^n \to \NN^n$  given by $\rr(\ab)=(r_i(a_i))_{1\leq i\leq n}$ where 
\[
r_i(a_i)=\left\{ 
\begin{array}{ll}
	t_i, & \text{ if } a_i>0,\\
	0, &  \text{ otherwise. }
\end{array}\right.
\]
It is easily seen that $\rr$ is an order preserving map, hence we may consider the functor $\rr^\ast: \ModN\to \ModN$ associated with $\rr.$

\begin{Proposition}\label{sqfree}
Let $M$ be a positively $\tb$-determined $\NN^n$-graded $S$-module. Then $\rr^\ast M$ is a positively $\One$-determined module, that is, 
$\rr^\ast M$ is a squarefree $S$-module.
\end{Proposition}

\begin{proof}
It is enough to show that, for any $\ab\in \NN^n$ and $i\in \supp(\ab),$ the multiplication map
\[
(\rr^\ast M)_{\ab} \stackrel{\cdot x_i}{\rightarrow}(\rr^\ast M)_{\ab+\eb_i}
\]
is an isomorphism of $K$-vector spaces. But this is almost obvious, since $\xb^{e_i}\cdot u=\xb^{\rr(\ab+\eb_i)-\rr(\ab)}u=u.$ The last equality is 
true  since $\supp(\xb^{\rr(\ab+\eb_i)-\rr(\ab)})=\emptyset$ for any $i\in \supp(\ab)$. Therefore, the multiplication by $x_i$ is the identity map, hence it is an isomorphism 
of vector  spaces.
\end{proof}

\begin{Proposition}\label{classicradical}
Let $\tb\in \NN^n$ with $\tb\geq \One,$ and let $I\subset S$ be a monomial ideal which is positively $\tb$-determined, that is, for every $u\in G(I),$ 
$\deg(u)\leq \tb.$ Then $\rr^\ast I\cong \sqrt{I}.$
\end{Proposition}

\begin{proof} 
Firstly, we claim that $\xb^\ab\in \sqrt{I}$ if and only if $\xb^{\rr(\ab)}\in I.$ Let us prove this claim. 
If $\xb^{\ab}\in \sqrt{I},$ then there exists $k\geq 1$ such that $\xb^{k\ab}\in I.$ Obviously, we may choose $k$ such that $ka_i\geq t_i$ for all 
$a_i>0.$ Then there exists $\xb^\bb\in G(I)$ such that $\xb^\bb\ |\ \xb^{k\ab},$ which implies that $\bb\leq k\ab.$ Since $I$ is positively 
$\tb$-determined, we also have $\bb\leq \tb.$ It then follows that $\bb\leq \rr(k\ab)=\rr(\ab)$ which implies that $\xb^\bb\ |\
\xb^{\rr(\ab)}$ and, therefore, 
$\xb^{\rr(\ab)}\in I.$

Conversely, let $\xb^{\rr(\ab)}\in I.$ We obviously may find $k\geq 1$ such that $k\ab\geq \rr(\ab),$ hence $\xb^{k\ab}\in I$ and, therefore, $\xb^{\ab}\in \sqrt{I},$ which ends the proof of our claim.

Let  $f:\sqrt{I}\to \rr^\ast I$ be the map given by $f=\oplus_{\ab}f_{\ab}$ where $f_{\ab}:(\sqrt{I})_\ab\to (\rr^\ast I)_{\ab}$ is defined by 
$f_{\ab}(\xb^\ab)=\xb^{\rr(\ab)}.$ The map  $f$ is obviously a graded isomorphism of $K$-vector spaces. We show that $f$ is an $S$-module isomorphism. Indeed,
for any $\ab, \bb,$ 
\[
f(\xb^\bb \cdot \xb^\ab)=\xb^{\rr(\ab+\bb)}= \xb^{\rr(\ab+\bb)-\rr(\ab)}\xb^{\rr(\ab)}= \xb^\bb\cdot f(\xb^\ab).
\]
\end{proof}

In order to state the first main result, we need a preparatory lemma. Before stating it, let us set some more notation. For $\ab\in \NN^n$, 
let  $S(-\ab)$ be the graded free $S$-module whose all graded components are obtained from those of $S$ by shifting with the vector $\ab,$ and 
let  $\sqrt{\ab}$  be the following vector of $\NN^n:$
\[
(\sqrt{\ab})_i=\left\{ 
\begin{array}{ll}
	1, & \text{ if } a_i>0,\\
	0, & \text{ if } a_i=0.
\end{array}\right.
\] 

\begin{Lemma}\label{shift}
Let $\tb\in \NN^n$ with $ \tb\geq \One.$ Then $r^\ast (S(-\ab))\cong S(-\sqrt{\ab})$ for every $\ab\in \NN^n$ with $\ab\leq \tb.$ 
\end{Lemma}

\begin{proof}
We obviously have the following isomorphisms:
\[
\rr^\ast(S(-\ab))\cong \rr^\ast(\xb^\ab)\cong (\xb^{\sqrt{\ab}})\cong S(-\sqrt{\ab}).
\]
\end{proof}

In the sequel we will always assume that  $\rr^\ast M\neq 0$. Note that $\rr^\ast M=0$ if and only if $M_{\ab}=0$ for all $\ab\in \NN^n$ such that  
$a_i\in\{0,t_i\}$ for $1\leq i\leq n.$

The following theorem shows that the graded Betti numbers go down when passing from the module to its radical. In particular, we may derive inequalities for the corresponding depths.

\begin{Theorem}\label{bettiradical}
Let $\tb\in \NN^n,$ $\tb\geq \One$, and let $M$ be a positively $\tb$-determined $\NN^n$-graded $S$-module. Then
\[
\beta_{i,\ab}(M)\geq \beta_{i,\sqrt{\ab}}(\rr^\ast M) 
\] 
for all $i$ and $\ab.$ In particular, the following inequality holds:
\[
\depth  M\leq \depth \rr^\ast M.
\]
\end{Theorem}

\begin{proof}
Let
\[
\FF_{\bullet}: \hspace{0.8cm}  0\to \bigoplus_{\ab}S(-\ab)^{\beta_{p,\ab}}\to \cdots \to \bigoplus_{\ab}S(-\ab)^{\beta_{1,\ab}} \to 
\bigoplus_{\ab}S(-\ab)^{\beta_{0,\ab}}\to M\to 0
\]
be a minimal free resolution of $M$ over $S.$ Since $M$ is positively $\tb$-determined, by \cite[Proposition 2.5]{Mill}, it follows that all the 
shifts in the above resolution are $\leq \tb.$ We apply $\rr^\ast$ to $\FF_{\bullet}$. By the exactness of  $\rr^\ast,$ we get a free 
$S$-resolution of $\rr^\ast M,$ possibly non-minimal. Therefore, we get the inequalities between the graded Betti numbers of $M$ and, respectively, 
$\rr^\ast M.$ These inequalities imply that $\projdim_S M\geq \projdim_S( \rr^\ast M)$ and, by using Auslander-Buchsbaum formula, we get the inequalities between  depths.
\end{proof}

\begin{Remark}\label{resol}{\em
By the above proof it follows that if
\[
\FF_{\bullet}:  0\to \bigoplus_{\ab}S(-\ab)^{\beta_{p,\ab}}\to \cdots \to \bigoplus_{\ab}S(-\ab)^{\beta_{1,\ab}} \to 
\bigoplus_{\ab}S(-\ab)^{\beta_{0,\ab}}\to M\to 0
\]
is a minimal $\ZZ^n$-graded free resolution of $M,$ then 
\[
\rr^\ast\FF_{\bullet}:  0\to \bigoplus_{\ab}S(-\sqrt{\ab})^{\beta_{p,\ab}}\to \cdots \to \bigoplus_{\ab}S(-\sqrt{\ab})^{\beta_{1,\ab}} \to 
\bigoplus_{\ab}S(-\sqrt{\ab})^{\beta_{0,\ab}}\to \rr^\ast M\to 0
\] 
is a free resolution of $\rr^\ast M$. Moreover, if the map $\partial_i: \bigoplus_{\ab}S(-\ab)^{\beta_{i,\ab}} \to 
\bigoplus_{\ab}S(-\ab)^{\beta_{i-1,\ab}}$ is given by the matrix $(\xb^{\ab_j-\bb_k})_{j,k}$, then $\rr^\ast \partial_i: \bigoplus_{\ab}S(-\sqrt{\ab})^{\beta_{i,\ab}} \to \bigoplus_{\ab}S(-\sqrt{\ab})^{\beta_{i-1,\ab}}$ is given by $(\xb^{\sqrt{\ab_j}-\sqrt{\bb_k}})_{j,k}$.
}
\end{Remark}

In the following corollary  we derive some consequences of Theorem~\ref{bettiradical}. To begin with, let $I\subset S$ be a monomial ideal. Then $I$ is a positively 
$\tb$-determined $\NN^n$-graded $S$-module if we choose, for instance, $\tb=(t_1,\ldots,t_n)$ where $t_i=\max\{\nu_i(u)\ |\ u\in G(I)\}$. Here   $G(I)$ denotes the minimal system of monomial generators of the ideal $I.$ Proposition~\ref{classicradical} says that $\rr^\ast I$ is actually $\sqrt{I}.$ Therefore, from Theorem~\ref{bettiradical}, we obtain as a consequence the following corollary which  extends results of \cite{HTT}.

\begin{Corollary}\label{extensionHHT}
Let $I\subset J\subset S$ be monomial ideals with $\sqrt{I}\neq \sqrt{J}$. Then $$\beta^S_{i,\ab}(J/I)\geq \beta^S_{i,\sqrt{\ab}}(\sqrt{J}/\sqrt{I})$$ for all $i\geq 0$ and $\ab\in 
\NN^n.$
Consequently, $$\depth_S(\sqrt{J}/\sqrt{I})\geq \depth_S(J/I).$$
\end{Corollary}

In the following we study the relationship between the Krull dimension of $M$ and $\rr^\ast M.$
We first introduce the following notation. For  $\ab\in \NN^n,$ $\supp(\ab)=\{i: a_i>0\}$ and $\supp^\tb(\ab)=\{i: a_i\geq t_i\}.$ We use the following convention. For $\ab, \bb \in \ZZ^n$,
let $\ab \cdot \bb$ denote the vector whose $i$-th component is
$a_ib_i$.

\begin{Proposition}\label{Ass}
Let  $M$ be a positively $\tb$-determined $S$-module where $\tb\in \NN^n, \tb\geq \One.$ Then $\Ass(\rr^\ast M)\subset \Ass(M).$
\end{Proposition}

\begin{proof}
Let $F\subset [n]$ and $P=P_F:=(x_i : i\notin F)$ be an associated prime of $\rr^\ast(M)$. Then, by \cite[Lemma 2.2]{Y}, there exists $0\neq u\in
(\rr^\ast M)_{\eb_F}$ such that $x_i u=0$ for all $i\notin F,$ where $\eb_F:=\sum_{i\in F}\eb_i.$ This means that there exists $0\neq u\in M_{\tb\cdot \eb_F}$ such that 
\[
\xb^{\rr(\eb_{F\cup\{i\}})-\rr(\eb_i)} u=\xb^{\tb\cdot \eb_i}u=x_i^{t_i}u=0
\]
for all $i\notin F.$ Then we may choose a maximal monomial (with respect to divisibility) $w\in K[\{x_i : i\notin F\}]$ such that $wu\neq 0.$ 
We claim that $P_F=\ann(wu)$, which will end the proof. 

If $i\notin F,$ then $x_i(wu)=0$ by the choice of $w,$ hence $P_F\subset\ann(wu).$ Let now $v$ be a monomial in $\ann(wu)$, that is, $v(wu)=0$. 
Clearly, for every monomial $w^\prime\in K[\{x_i: i\in F\}]$, we have $w^\prime wu\neq 0$ since $\supp(w^\prime)\subset \supp^\tb(wu)$ and $M$ is positively $\tb$-determined. This implies that there exists $i\notin F$ such that $x_i$ divides $v$, thus $v\in P_F.$
\end{proof}

\begin{Corollary}\label{dim}
Let $M$ be a positively $\tb$-determined module. Then $\dim M\geq \dim \rr^\ast M.$
\end{Corollary}
Note that the inequality $\dim \rr^\ast M\leq \dim M$ may be strict as the following example shows. On the other hand, we have 
$\dim \rr^\ast M = \dim M$ if and only if there exists $\ab\in \NN^n$ such that $\#\supp^\tb(\ab)=\dim M$ and $M_{\rr(\ab)}\neq 0$.

\begin{Example}{\em 
Let $I=(a^4d^4,a^2b^3,b^3c^2,b^3d)$ and $J=(a^3d^3,a^3b,b^2)$, $I,J\subset K[a,b,c,d]$. One may easily check  that 
$\dim(\sqrt{J}/\sqrt{I})=1 < \dim(J/I)=2.$
}
\end{Example}

Let us recall that a finitely generated $S$-module is called {\em equidimensional} if all its minimal primes have the same codimension. As an immediate  consequence of Proposition~\ref{Ass} we get also the following 

\begin{Corollary}\label{equidim}
Let $M$ be a positively $\tb$-determined module such that $\dim M=\dim \rr^\ast M.$ If $M$ is equidimensional, then $\rr^\ast M$ is equidimensional, too.
\end{Corollary}

The following example shows that  the implication of the above corollary is no longer true if $\dim M > \dim \rr^\ast M$.

\begin{Example}{\em
Let $P,P_1,P_2\subset S,$ $P=(x_1), P_2=(x_1,x_2)$, $P_2=(x_1,x_3,x_4),$ and $M=(S/P)(-(1,0,\ldots,0))\oplus (S/P_1)\oplus (S/P_2)$.
Then $M$ is positively $\tb$-determined, where $\tb=(2,1,\ldots,1),$ and equidimensional. We have $\rr^\ast M=(S/P_1)\oplus (S/P_2)$, thus $\rr^\ast M$ is not equidimensional. But, of course, $\dim M > \dim \rr^\ast M.$
}
\end{Example}

In \cite{HTT} it is shown that if $S/I$ is (sequentially) Cohen-Macaulay, then $S/\sqrt{I}$ shares the same property. We are going to extend this 
result to any positively determined $\NN^n$-graded module. 

\begin{Theorem}\label{CM}
Let  $M$ be a positively $\tb$-determined $S$-module with $\rr^\ast M \neq 0$ where $\tb\in \NN^n, \tb\geq \One.$ 
\begin{itemize}
	\item [(i)] If $M$ is Cohen-Macaulay, then $\rr^\ast M$ is Cohen-Macaulay and $\dim \rr^\ast M=\dim M.$
	\item [(ii)] If $M$ is sequentially Cohen-Macaulay, then $\rr^\ast M$ is sequentially Cohen-Macaulay.
\end{itemize}
\end{Theorem}

\begin{proof} (i). By Theorem~\ref{bettiradical} and Corollary~\ref{dim}, we get the following inequalities:
\[
\depth M\leq \depth \rr^\ast M\leq \dim \rr^\ast M\leq \dim M.
\] Since $M$ is Cohen-Macaulay, we get the desired conclusions.

(ii). As $M$ is sequentially Cohen-Macaulay, there exists a finite filtration 
\[
0=M_0\subset M_1\subset \cdots \subset M_r=M
\]  by graded submodules of $M$ such that each quotient $M_i/M_{i-1}$ is Cohen-Macaulay and 
\[
\dim M_1/M_0 < \dim M_2/M_1 <\cdots <\dim M_r/M_{r-1}.
\] This filtration induces the following filtration of $\rr^\ast M,$
\[
0=\rr^\ast M_0\subset \rr^\ast M_1\subset \cdots \subset \rr^\ast M_r=\rr^\ast M.
\]
By (i), each factor in this filtration is either $0$ or a Cohen-Macaulay module with  $\dim \rr^\ast M_i/\rr^\ast M_{i-1}=\dim  M_i/M_{i-1}.$  By skipping the redundant factors in the above filtration we get the desired filtration for $\rr^\ast M.$ Therefore, (ii) follows.
\end{proof}

We may now derive the following corollary which  extends some results of \cite{HTT}.

\begin{Corollary} \label{CMextend}
Let $I\subset J\subset S$ be monomial ideals such that $\sqrt I\neq \sqrt J$. Then:
	\item [(i)] If $J/I$ is Cohen-Macaulay, then $\sqrt J/\sqrt I $ is Cohen-Macaulay and, moreover, 
	 $$\dim J/I=\dim \sqrt J/\sqrt I.$$
	\item [(ii)] If $J/I$ is sequentially Cohen-Macaulay, then $\sqrt J/\sqrt I$ is sequentially Cohen-Macau\-lay.
\end{Corollary}

 \section{The functor $\rr^\ast$  and $\Ext$}
 \label{extsection}

For $M \in \ModZ$ and $\ab \in \ZZ^n$,
$M(\ab)$ denotes the $\ZZ^n$-graded $S$-module such that $M = M(\ab)$
as underlying $S$-modules and the degree is given by the formula
$M(\ab)_\bb = M_{\ab + \bb}$. 
Following the usual convention, for $M, N \in \ModZ$,
we set $\grHom_S(M,N) := \bigoplus_{\ab\in \ZZ^n}\Hom_{\ModZ}(M,N(\ab))$.
Note that if $M$ is finitely generated, $\grHom_S (M,N) = \Hom_S(M,N)$.
Let $\grExt^i_S(-,N)$ (resp. $\grExt^i_S(M,-)$) be the $i$-th right derived functor of $\grHom_S(-,N)$ (resp. $\grHom_S(M, -)$).

In this section, we will study the relation between  $\rr^\ast$ and $\grExt$-functor. For this purpose, we need the following three functors.

Recall that degree-shifting induces an endofunctor of $\ModZ$.
For $\ab \in \ZZ^n$, let $\sigma_\ab$ denote the functor given by shifting degree by $\ab$. Thus we have $\sigma_\ab(M) = M(-\ab)$ for all $M \in \ModZ$.

For $M \in \ModZ$ and $\ab \in \ZZ^n$,
the truncated module $\tau_{ \ab}(M) := \bigoplus_{\bb \ge \ab} M_\bb$ is again an object of $\ModZ$, and any morphism $M \to N$ in $\ModZ$ induces the one $f|_{\tau_{ \ab}(M)}:\tau_{ \ab}(M) \to \tau_{ \ab}(N)$.
Thus we have the functor $\tau_{ \ab}: \ModZ \to \ModZ$.

Let $\ss : \NN^n \to \ZZ^n$ be the function defined by $\ss(\ab) = (s_i(a_i))_{1 \le i \le n}$ where
$$
s_i(a_i) = \begin{cases}
t_i &\text{if $a_i \ge 1$,}\\
t_i - 1  &\text{otherwise.}
\end{cases}
$$
The induced functor $\ss^\ast: \ModZ \to \ModN$, if restricted to $\Modtb$, is an endofunctor of $\Modtb$.

For $M \in \Modtb$ and $\ab \in \ZZ^n$ with $\ab \ge \Zero$,
the multiplication
$$
x^{\ab \cdot \tb - \rr(\ab)}: M_{\rr(\ab)} \to M_{\ab \cdot \tb}
$$
is a $K$-linear isomorphism since $\supp(\ab\cdot\tb-\rr(\ab))\subset \supp^\tb\rr(\ab)$. Let $\phi^M_\ab$ denote this map.

Now we are ready to define the two natural transformations $\Phi^\rr: \id_{\Modtb} \Longrightarrow \rr^\ast$ between endofunctors of $\Modtb$ and
$\Psi: \ss^\ast \Longrightarrow \sigma_{\One - \tb}$ between those of $\ModN$.
For $M \in \Modtb$, let $\Phi_M : M \to \rr^\ast(M)$ be the
map defined as follows; for a homogeneous $u \in M_\ab$ with $\ab \ge \Zero$,
$$
\Phi_M(u) := {(\phi^M_{\ab})}^{-1}(x^{\ab \cdot (\tb - \One)}u) \in \rr^\ast(M)_\ab.
$$

For $\ab \in \NN^n$, it is easy to verify that $\ab + \tb - \One \ge \ss(\ab)$.
For $M \in \ModN$, we define the map $\Psi_M: \ss^\ast(M) \to \sigma_{\One - \tb}(M)$ as follows; for $\ab \in \NN^n$ and a homogeneous $u \in \ss^\ast(M)_\ab$ with $\ab \ge \Zero$,
$$
\Psi_M(u) = x^{\ab + \tb - \One - \ss(\ab)} \cdot u \in \sigma_{\One - \tb}(M)_\ab.
$$

\begin{Lemma}\label{lem:nat}
The following statements hold.
\begin{enumerate}
\item The above $\Phi$ is indeed a natural transformation from $\id_{\Modtb}$ to $\rr^\ast$.
\item Let $\iota: \Modone \to \Modtb$ be the canonical embedding.
Then $\Phi$ induces the natural isomorphism $\id_{\Modone} \Longrightarrow \rr^\ast \circ \iota$ between endofunctors of $\Modone$.
\item The above $\Psi$ is indeed a natural transformation from $\ss^\ast$ to $\sigma_{\One - \tb}$.
\end{enumerate}
\end{Lemma}
\begin{proof}
We will prove only the assertions (1) and (2). The rest is proved by the way similar to (1).

(1) First, we must verify that $\Phi_M$ is an $S$-linear map.
Take any $u \in M_\ab$ and $\bb \in \ZZ^n$ with $\bb \ge \Zero$.
Then
\begin{align*}
\phi^M_{\ab + \bb} (x^\bb \cdot \Phi_M(u))
  & = x^{(\ab + \bb ) \cdot \tb - \rr(\ab + \bb)} (x^{\rr(\ab + \bb) - \rr(\ab)} {(\phi^M_{\ab})}^{-1}(x^{\ab \cdot (\tb - \One)}u)) \\
  & = x^{(\ab + \bb) \cdot \tb - \rr(\ab)} {(\phi^M_{\ab})}^{-1}(x^{\ab \cdot (\tb - \One)}u) \\
  & = x^{\bb \cdot \tb} (x^{\ab \cdot \tb - \rr(\ab)} {(\phi^M_{\ab})}^{-1}(x^{\ab \cdot (\tb - \One)}u))
    = x^{\bb \cdot \tb} \cdot x^{\ab \cdot (\tb -\One)}u\\
  &  = x^{(\ab + \bb)\cdot (\tb - \One)} (x^\bb u),
\end{align*}
and hence it follows that
$$
x^\bb \cdot \Phi_M(u) = (\phi^M_{\ab + \bb})^{-1}(x^{(\ab + \bb) \cdot (\tb - \One)} (x^\bb u)) = \Phi_M(x^\bb u).
$$
Thus $\Phi_M$ is indeed $S$-linear.

Next let $f: M \to N$ be a morphism in $\Modtb$. We will show that the following diagram
commutes;
$$
\begin{CD}
M @>\Phi_M >> \rr^\ast(M) \\
@VfVV  @VV\rr^\ast(f)V \\
N @>>\Phi_N> \rr^\ast(N).
\end{CD}
$$
Let $u \in M_\ab$. Then
\begin{align*}
\phi^N_\ab(\rr^\ast(f) \circ \Phi_M(u)) &= x^{\ab \cdot \tb - \rr(\ab)} \cdot f (\Phi_M(u)) \\
  &= f(x^{\ab \cdot \tb - \rr(\ab)} \cdot \Phi_M(u) ) \\
  &= f(x^{\ab \cdot \tb - \rr(\ab)} \cdot (\phi^M_\ab)^{-1}(x^{\ab \cdot (\tb - \One)} u)) \\
  &= f(x^{\ab \cdot (\tb - \One)} u) = x^{\ab \cdot(\tb - \One)} \cdot f(u)
\end{align*}
Therefore we conclude that
$$
\rr^\ast(f) \circ \Phi_M(u) = (\phi^N_\ab)^{-1}(x^{\ab \cdot (\tb - \One)} \cdot f(u))
= \Phi_N \circ f (u).
$$

(2) Let $M \in \Modone$. We have to show $\Phi_M: M \to \rr^\ast(M)$ is then an isomorphism. Since both of $M$ and $\rr^\ast(M)$ are objects in $\Modone$,
it suffices to show that each $(\Phi_M)_\ab: M_\ab \to \rr^\ast(M)_\ab$ is an isomorphism for all $\ab$ with $\Zero \le \ab \le \One$.
This is an immediate consequence of the fact that for such $\ab$, the multiplication map
$$
M_\ab \xrightarrow{x^{\ab \cdot (\tb - \One)}} M_{\ab \cdot \tb}
$$
is an isomorphism since $M \in \Modone$.
Thus we conclude that $\Phi_M$ is an isomorphism for all $M \in \Modone$.
\end{proof}

%
%
%
%
%
%
%
%
%
 Let $\Dc$ denote the contravariant functor $\grHom_S(-, S): \ModZ \to \ModZ$.
We set $\Dc_{\tb} := \sigma_{\tb } \circ \Dc$. Note that $\Dc_{\tb}$ gives a duality
on $\Modtb$, and $\Dc_{\One} = \sigma_{-\tb + \One} \circ \Dc_{\tb}$
is the usual duality on $\ModZ$ by the canonical module $S(-\One)$ of $S$.
The functor $\Dc_{\tb}$, lifted up to a functor from the category of complexes  in $\Modtb$ to itself, coincides with the one $\Dc_\tb$ in \cite{BF} up to shifting and quasi-isomorphism \cite[Proposition 3.6]{BF}.
Moreover $\Dc_\tb$ (resp. $\Dc_\One$) sends $M \in \Modtb$ (resp. $M \in \Modone$) to an object in $\Modtb$ (resp. $\Modone$).

\begin{Proposition}\label{prop:r_and_D}
There exists the natural isomorphisms between functors
\begin{enumerate}
\item $\tau_{\Zero} \circ \Dc_\One \simeq \Dc_\One \circ \rr^\ast$ and
\item $\rr^\ast \circ \Dc_\tb \simeq \Dc_\One \circ \ss^\ast$,
\end{enumerate}
from $\Modtb$ to $\Modone$.
\end{Proposition}
\begin{proof}
(1) By Lemma \ref{lem:nat}, there exists the natural transformation $\Phi:\id_{\Modtb} \Longrightarrow \rr^\ast$, and hence we have the one
$\tau_{\Zero} \circ \Dc_{\One} \circ \rr^\ast \Longrightarrow \tau_{\Zero} \circ \Dc_{\One}$, where both  functors are regarded as the ones from $\Modtb$ to $\ModN$.
Since $\rr^\ast M$ is a squarefree module,
it follows that $\tau_{\Zero} \circ \Dc_{\One} \circ \rr^\ast = \Dc_{\One} \circ \rr^\ast$. Consequently, the above natural transformation induces
the one $\Phi': \Dc_{\One} \circ \rr^\ast \Longrightarrow \tau_{\Zero} \circ \Dc_{\One}$ of functors from $\Modtb$ to $\ModN$.

Note that any $M \in \Modtb$ has a presentation
$$
F_1 \longto F_0 \longto M \longto 0
$$
with $F_0, F_1$ free modules given by direct sums of finitely many copies of
$S(-\ab)$ with $\Zero \le \ab \le \tb$. Since the functors $\rr^\ast, \tau_{\tb - \One}$ are exact and since $\Dc_\tb$ is left exact, we have the following
commutative diagram with exact rows;
$$
\begin{CD}
0 @>>> \Dc_\One \circ \rr^\ast(M) @>>> \Dc_\One \circ \rr^\ast(F_0) @>>> \Dc_\One \circ \rr^\ast(F_1) \\
@.  @V{\Phi'}_M VV @V{\Psi'}_{F_0}VV @V{\Phi'}_{F_1}VV \\
0 @>>> \tau_{\Zero} \circ \Dc_\One (M) @>>> \tau_{\Zero} \circ \Dc_\One(F_0) @>>> \tau_{\Zero} \circ \Dc_\One(F_1).
\end{CD}
$$
Thus what we have to show is that ${\Phi'}_{S(-\ab)}$ is an isomorphism for any
$\ab \in \ZZ^n$ with $\Zero \le \ab \le \tb$.

Let $N$ be the cokernel of ${\Phi'}_{S(-\ab)}: S(-\ab) \to S(-\sqrt{\ab}) = \rr^\ast(S(-\ab))$. The map ${\Phi'}_S(-\ab)$ is obviously injective, and hence
the following sequence
\[
0\to S(-\ab)\to \rr^\ast(S(-\ab))=S(-\sqrt{\ab})\to N\to 0,
\]
is exact. By applying $\Dc_{\One}$, we obtain:
\[
0\to \Dc_{\One}(N) \to \Dc_{\One}\circ \rr^\ast(S(-\ab)) \to \Dc_{\One}(S(-\ab))
\]
 As obviously $\dim N\leq n-1,$ it follows that   $\Dc_\One(N) = 0$. Therefore we get the following exact sequence
\[\begin{CD}
0 @>>> \Dc_\One(S(-\sqrt{\ab})) @>{\Phi'}_{S(-\ab)}>> \tau_{\Zero} \circ \Dc_\One(S(-\ab)).
\end{CD}
\]
Now $\Dc_\One(S(-\sqrt{\ab})) \cong S(-\One + \sqrt{\ab})$ and $\Dc_\One(S(-\ab)) \cong S(-\One  + \ab)$.
Easy observation shows that for $\bb \in \ZZ^n$ with $\bb \ge \Zero$,
$S(-\One + \sqrt{\ab})_\bb \neq 0$ if and only if $\tau_{\Zero}(S(-\One + \ab))_\bb \neq 0$.
Therefore it follows that ${\Phi'}_{S(-\ab)}$ is an isomorphism.

(2) The assertion (2) can be proved in the way similar to (1).
We have the natural transformation
$$
\rr^\ast \circ \Dc_\tb = \rr^\ast \circ \Dc_\One \circ \sigma_{\One - \tb} \Longrightarrow \rr^\ast \circ \Dc_\One \circ \ss^\ast
$$
between functors from $\Modtb$ to $\Modone$ by Lemma 2.1.
Since $\Dc_\One \circ \ss^\ast$ sends $M \in \Modtb$ to an object in $\Modone$,
there is a natural isomorphism $\rr^\ast \circ \Dc_\One \circ \ss^\ast \simeq \Dc_\One \circ \ss^\ast$ by Lemma 2.1 again.
Consequently we have the natural transformation $\Psi':\rr^\ast \circ \Dc_\tb \Longrightarrow \Dc_\One \circ \ss^\ast$.
By the argument similar as above, we have only to prove ${\Psi'}_{S(-\ab)}$ is an isomorphism for all $\ab$ with $\Zero \le \ab \le \tb$.
By the definition of $\Psi$, $\Psi_{S(-\ab)}$ is injective. Applying $\rr^\ast \circ \Dc_\One$ to the exact sequence
$$
0 \longto \ss^\ast(S(-\ab)) \xrightarrow{\Psi_{S(-\ab)}} \sigma_{\One - \tb}(S(-\ab)) \longto M \longto 0,
$$
where $M$ is the cokernel of $\Psi_{S(-\ab)}$,
we have the exact one
$$
0 \longto \rr^\ast \circ \Dc_\tb(S(-\ab)) \xrightarrow{{\Psi'}_{S(-\ab)}}\Dc_\One \circ \ss^\ast(S(-\ab))
$$
We define $\bb = (b_i)_{1 \le i \le n}$ by setting $b_i = 0$ if $a_i \le t_i - 1$ and
$b_i = 1$ if  $a_i = t_i$.
It follows that $\ss^\ast(S(-\ab)) \cong S(-\bb)$, and hence $\Dc_\One \circ \ss^\ast(S(-\ab)) \cong S(-(\One - \bb))$.
On the other hand, $\rr^\ast \circ \Dc_\tb(S(-\ab)) \cong S(-\sqrt{\tb - \ab})$.
Easy calculation shows $\One - \bb = \sqrt{\tb - \ab}$, and therefore
${\Psi'}_{S(-\ab)}$ is an isomorphism.
\end{proof}

\begin{Theorem} \label{ext}
The following statements hold.
\begin{enumerate}
\item Let $\omega_S$ be the canonical module of $S$. Then
there are the following two isomorphisms
\begin{enumerate}
\item $\tau_{\Zero} \grExt^p_S(M, \omega_S) \cong \grExt^p_S(\rr^\ast(M), \omega_S)$
\item $\rr^\ast(\grExt^p_S(M, S(-\tb))) \cong \grExt^p_S(\ss^\ast(M),\omega_S)$
\end{enumerate}
for all $p$ and $M \in \Modtb$.
\item Assume $I$ is a Cohen-Macaulay monomial ideal such that
$S/I \in \Modtb$. Let $\omega_{S/I}$, $\omega_{S/\sqrt I}$ be
the canonical modules of $S/I$, $S/\sqrt I$, respectively. Then
$$
\omega_{S/\sqrt I} \cong \tau_{ \Zero}(\omega_{S/I}) = (\omega_{S/I})_{\ge \Zero}.
$$
\end{enumerate}
\end{Theorem}

\begin{proof}
Choose a $\ZZ^n$-graded minimal free resolution $P_\bullet$ of $M$
with each $P_i$ positively $\tb$-determined.
By Proposition~\ref{prop:r_and_D}, 
$$
\tau_\Zero \grExt^p_S(M,\omega_S)
   \cong H^p(\tau_\Zero \circ \Dc_\One(P_\bullet))
   \cong H^p(\Dc_\One \circ \rr^\ast(P_\bullet))
   \cong \grExt^p(\rr^\ast(M), \omega_S),
$$
since $\rr^\ast(P_\bullet)$ is a free resolution of $\rr^\ast(M)$.

The assertion (2) follows from (1) and Theorem~\ref{CM}.
\end{proof}

\begin{Remark}{\em 
In \cite[Corollary 2.3]{HTT}, Herzog, Takayama, and Terai proved
$$
H^i_\mm(S/I)_\ab \cong H^i_\mm(S/\sqrt{I})_\ab
$$
for any monomial ideal $I$ and all $i$ and $\ab \le \Zero$.
In (1) of Theorem \ref{ext}, taking the graded Matlis duality, we obtain the generalization of this result, which also implies that there exists the isomorphism
$$
H^i_\mm(S/I)_{\le \Zero} \cong H^i_\mm(S/\sqrt I)
$$
as $\ZZ^n$-graded $S$-modules. Here, for any $M \in \ModZ$, $M_{\le \Zero}$ denotes the residue of $M$ by its $\ZZ^n$-graded submodule generated by the homogeneous elements whose degree is not less than or equal to $\Zero$.}
\end{Remark}

As an immediate consequence of the above theorem we get the following corollary that generalizes a result of \cite{HTT}. 

\begin{Corollary}\label{gCM}
Let $M$ be a positively $\tb$-determined  $S$-module with $\dim M=\dim \rr^\ast M.$ Then:
\begin{enumerate}
	\item $M$ is generalized Cohen-Macaulay if and only if so is $\rr^\ast M$.
	\item If $M$ is Buchsbaum, then $\rr^\ast M$ is Buchsbaum.
\end{enumerate}
\end{Corollary}

\begin{proof}
(1) $M$ is generalized Cohen-Macaulay if and only if  $\grExt^i_S(M,\omega_S)$ has finite length for any $i\neq n-d$, where $d=\dim M$. This is equivalent to say that $\tau_0 \grExt^i_S(M,\omega_S)$ has finite length for $i\neq n-d$. Therefore, since $M$ and $\rr^\ast M$ have the same dimension, the desired statement follows by Corollary~\ref{ext}.

(2) If $M$ is Buchsbaum, then $M$ is generalized Cohen-Macaulay \cite{Go}, thus $\rr^\ast M$ is generalized Cohen-Macaulay. By \cite[Corollary 2.7]{Y},
it follows that $\rr^\ast M$ is Buchsbaum.
\end{proof}
 
 \section{The $\rr^\ast$ functor and Alexander duality}

Recall that there exists the duality $\Ac_{\tb}$ on $\Modtb$, called Alexander duality, defined by Miller \cite{Mill}. In the case, $\tb = \One$,
R\"{o}mer also defines independently in \cite{R}.
Let $E$ denote the injective hull of $K$ and $p_\tb: \NN^n \to \ZZ^n$
the map defined by $p_\tb(\ab) := (p_{t_1}(a_1),\dots ,p_{t_n}(a_n))$,
where
$$
p_{t_i}(a_i) := \begin{cases}
a_i & \text{if $0 \le a_i < t_i$,} \\
t_i & \text{if $a_i \ge t_i$.}
\end{cases}
$$
The functor $\Ac_{\tb}$ is given by the formula
$$
\Ac_{\tb}(M) = p_{\tb}^\ast\grHom_S(M(\tb),E).
$$
If we set $\DD$ to be the functor $\grHom_K(-, K)$, we have the natural isomorphism $\Ac_\tb \cong p_\tb^\ast \circ \sigma_\tb \circ \DD$.

The following is essentially proved by Miller in \cite{Mill}.

\begin{Proposition}\label{prop:r_and_A}
There exists a natural isomorphism of functors from $\Modtb$ to itself
$$
\Ac_\One \circ \rr^\ast \simeq \rr^\ast \circ \Ac_{\tb}.
$$
\end{Proposition}
\begin{proof}
Let $M \in \Modtb$ and $\ab \in \ZZ^n$ with $\ab \ge \Zero$.
\begin{align*}
\Ac_\One \circ \rr^\ast (M)_\ab &= \grHom_K( \rr^\ast(M), K)_{p_\One(\ab) - \One}
                                             = \grHom_K(M_{\rr(- p_\One(\ab) + \One)}, K), \\
\rr^\ast \circ \Ac_\tb (M)_\ab &= \grHom_K(M, K)_{p_\tb(\rr(\ab)) - \tb}
                                            = \grHom_K(M_{-p_\tb(\rr(\ab)) + \tb}, K).
\end{align*}
It is easy to verify that $\rr(- p_\One(\ab) + \One) = -p_\tb(\rr(\ab)) + \tb$.
Hence $\Ac_\One \circ \rr^\ast(M) \cong \rr^\ast \circ \Ac_\tb(M)$
as $\ZZ^n$-graded $K$-vector spaces. By a routine calculation, we can show that
this isomorphism is $S$-linear and natural.
\end{proof}
 
\section{Relation to the Auslander-Reiten translate}

Let $\le$ be the order on $\ZZ^n$ defined as follows:
$$
\ab \le \bb \ \Longleftrightarrow \ a_i \le b_i
$$
for all $i$. With the order induced by $<$,
the set $P_\tb := \{ \ab \in \ZZ^n : \Zero \le \ab \le \tb \}$ becomes a poset.
Let $A$ denote the incidence algebra of $P_\tb$ over $K$.
It is well-known that the category $\Modtb$ is equivalent to the one $\mod A$ consisting of finite-dimensional left $A$-modules (\cite[3.5]{BF} and \cite[Proposition 4.3]{Y2}). Since $A$ is a finite-dimensional $A$-algebra of finite global dimension, as is shown in \cite[3.6]{H} by Happel, the bounded derived category $D^b(\mod A)$ of $\mod A$ has Auslander-Reiten triangles.
Let $K^b(\Proj A)$ (resp. $K^b(\Inj A)$) be the bounded homotopy category of complexes of finite-dimensional projective (resp. injective) left modules.
According to Happel's proof, through the equivalence $K^b(\Proj A) \cong D^b(\mod A) \cong K^b(\Inj A)$ of triangulated categories, the Auslander-Reiten translate (see \cite[Definition 1.2]{J} for its definition) is then given by $T^{-1} \circ v$, where $T$ is the usual translation functor and $v$ is the equivalence $K^b(\Proj A) \cong K^b(\Inj A)$ of triangulated categories induced by the Nakayama functor, i.e., $\Hom_K(\Hom_A(-,A),K)$.
On the other hand $v$ is coincides with $\Ac_\tb \circ \Dc_\tb$ through the equivalence $\Modtb \cong \mod A$. In this sense, $\Ac_\tb \circ \Dc_\tb$ represents the Auslander-Reiten translate of $D^b(\mod A)$ (see\cite[3.3 and 3.5]{BF} for details).

In this section, we discuss the relation between $\rr^\ast$ and $\Ac_\tb \circ \Dc_\tb$. For this, we need to define a new functor.
For $\ab \in \ZZ^n$, let $\tau^\ab$ be the functor from $\ModZ$ to $\ModZ$ defined as follows: for any $M \in \ModZ$, we set
$$
\tau_\ab(M) := M/(S \cdot \bigoplus_{\bb \not\le \ab} M_\ab),
$$
and for a morphism $f: M \to N$ in $\ModZ$, we assign, $\tau_\ab(f)$, the natural homomorphism induced by $f$.
Clearly $\tau_\ab$ is additive and exact.

\begin{Lemma}\label{lem:sig_tau_2}
There are the following natural isomorphisms of functors:
for any $\ab$, $\bb \in \ZZ^n$,
\begin{enumerate}
\item $\sigma_\ab \circ \tau^\bb \simeq \tau^{\ab + \bb} \circ \sigma_\bb$,
\item $\DD \circ \sigma_\ab \simeq \sigma_{-\ab} \circ \DD$, and
\item $\DD \circ \tau_\ab \simeq \tau^\ab \circ \DD$.
\end{enumerate}
\end{Lemma}
\begin{proof}
By comparing each degree $\cb$ component, for each $M \in \ModZ$, we obtain an isomorphism, as $\ZZ^n$-graded $K$-vector spaces, between two modules given by applying each of two functors. An easy calculation shows these maps indeed defines the desired natural isomorphisms.
\end{proof}

\begin{Proposition}\label{ART}
There exists the following two natural isomorphisms of functors
\begin{enumerate}
\item $\Ac_\One \circ \Dc_\One \circ \rr^\ast \simeq p^\ast_\One \circ \Ac_\tb \circ \Dc_\tb$
\item $\rr^\ast \circ \Ac_\tb \circ \Dc_\tb \simeq \Ac_\One \circ \Dc_\One \circ \ss^\ast$
\end{enumerate}
from $\Modtb$ to $\Modone$.
\end{Proposition}
\begin{proof}
The second natural isomorphism is a direct consequence of Propositions \ref{prop:r_and_D} and \ref{prop:r_and_A}.
We will show the first. It follows from Proposition~\ref{prop:r_and_D} and Lemma~\ref{lem:sig_tau_2} that
\begin{align*}
\Ac_\One \circ \Dc_\One \circ \rr^\ast
  &\simeq \Ac_\One \circ \tau_\Zero \circ \Dc_\One \\
  &\simeq p^\ast_\One \circ \sigma_\One \circ \DD \circ \tau_\Zero \circ \sigma_{\One - \tb} \circ \Dc_\tb \\
  &\simeq p^\ast_\One \circ \sigma_{\One} \circ \tau^{\Zero} \circ \sigma_{\tb - \One} \circ \DD \circ \Dc_\tb \\
  &\simeq p^\ast_\One \circ \tau^\One \circ \sigma_\tb \circ \DD \circ \Dc_\tb.
\end{align*}
By the definition, it is easy to verify that $p^\ast_\One \circ \tau^\One = p^\ast_\One$. Moreover $p_\tb \circ p_\One = p_\One$ implies $p^\ast_\One \circ p^\ast_\tb = p^\ast_\One$.
Thus it follows that
$$
\Ac_\One \circ \Dc_\One \circ \rr^\ast
  \simeq p^\ast_\One \circ \tau^\One \circ \sigma_\tb \circ \DD \circ \Dc_\tb 
  \simeq p^\ast_\One \circ \Ac_\tb \circ \Dc_\tb.
$$
\end{proof}
 
{}
 
 \end{document}